\documentclass[11pt]{amsart}

\usepackage[margin=1.15in]{geometry}

\usepackage{multicol}

\usepackage{graphicx}
\usepackage{amsmath, amssymb, amsthm}
\usepackage{pdflscape}
\usepackage{longtable}
\usepackage{float}
\usepackage{tikz}
\usepackage{tikz-3dplot}
\usepackage{subfigure}
\usepackage{mathtools}

\everydisplay{\small}

\usepackage{hyperref}

\hypersetup{
    colorlinks=true,
    linkcolor=blue,
    filecolor=magenta,
    urlcolor=cyan,
    citecolor=magenta,
    pdftitle={Normalization of Puiseux hypersurfaces},
    pdfpagemode=FullScreen,
}

\usepackage[citestyle=alphabetic, bibstyle=alphabetic]{biblatex}
\usepackage{color}

\theoremstyle{plain}

\newtheorem{thm}{Theorem}[section]

\newtheorem{cor}[thm]{Corollary}
\newtheorem{lem}[thm]{Lemma}
\newtheorem{prop}[thm]{Proposition}

\theoremstyle{definition}

\theoremstyle{remark}

\newtheorem{rem}[thm]{Remark}

\newtheorem{exa}[thm]{\textbf{Example}}

\theoremstyle{definition}

\newtheorem{defn}[thm]{\textbf{Definition}}

\theoremstyle{remark}

\renewcommand{\deg}{\text{deg}\,}

\renewcommand{\phi}{\varphi}

\begin{document}


\title[]{Normalization of Puiseux hypersurfaces}

\author[]{Fuensanta Aroca}
\email{fuen@im.unam.mx}
\address{
Instituto de Matem\'aticas, Universidad Nacional Aut\'onoma de M\'exico\\
Laboratorio Solomon Lefschetz, Unidad Mixta Internacional CRNS-UNAM-SECIHTI
}

\author[]{Annel Ayala}
\email{annel\_ayala@im.unam.mx}
\address{
Instituto de Matem\'aticas, Universidad Nacional Aut\'onoma de M\'exico
}

\author[]{Oscar Castañon}
\email{oscar.moreno@im.unam.mx}
\address{
Instituto de Matem\'aticas, Universidad Nacional Aut\'onoma de M\'exico
}

\author[]{Diana Mendez Penagos}
\email{diana.mendez@im.unam.mx}
\address{
Instituto de Matem\'aticas, Universidad Nacional Aut\'onoma de M\'exico
}

\author[]{Damian Ochoa}
\email{damian.ochoa@im.unam.mx}
\address{
Instituto de Matem\'aticas, Universidad Nacional Aut\'onoma de M\'exico
}
\address{
Escuela de Ciencias Físicas y Matem\'aticas,
Universidad de San Carlos de Guatemala
}
\thanks{Corresponding author. E-mail: damian.ochoa@im.unam.mx}

\author[]{Camille Plénat}
\email{camille.plenat@univ-amu.fr}
\address{
Aix Marseille Univ, CNRS, I2M.
}


\keywords{
Puiseux hypersurfaces, quasi-ordinary singularities,
toric varieties, Hirzebruch-Jung singularities
}

\thanks{
{\it 2020 Mathematics Subject Classification}: 32S05, 14B05, 13F25, 14M25.\\
\mbox{\hspace{11pt}}
This work has been partially supported by Laboratorio Internacional
Solomon Lefschetz; by DGAPA--Universidad Nacional Autónoma de México
through PAPIIT IN113323 and IN112626; and by SECIHTI, Mexico,
CVUs: 1337652, 781791, 1044169.
We are grateful to M. Spivakovsky for his helpful comments and suggestions,
which greatly improved this work.
The corresponding author also thanks P. Popescu-Pampu and
P. González-Pérez for providing valuable references.
}


\begin{abstract}
It is known that the normalization of a quasi-ordinary complex singularity
is a Hirzebruch-Jung, see
\cite{perez2000singularites, popescu2004analytical, aroca2005normal}.
We extend this result to Puiseux hypersurfaces.
Moreover, we prove that Hirzebruch-Jung singularities are precisely
normalizations of Puiseux hypersurfaces.
Our result holds over an algebraically closed field whose characteristic
does not divide the degree of the polynomial defining the hypersurface.
Finally, in the analytic complex case, we conclude that the normalization
of an irreducible Puiseux hypersurface is the normalization of a complex
analytic quasi-ordinary singularity.
\end{abstract}

\maketitle


\begin{multicols}{2}

\section*{Introduction}\label{sec1}

\noindent
We denote by $\mathbb{K}$ an algebraically closed field.
The ring of Puiseux series is defined as
\[
\resizebox{\columnwidth}{!}{$\displaystyle
\mathbb{K}[[X^*]]
=
\mathbb{K}[[x_1^*, \ldots, x_n^*]]
:=
\bigcup_{m\in\mathbb{Z}_{>0}}
\mathbb{K}\left[\!\left[
x_1^{1/m}, \ldots, x_n^{1/m}
\right]\!\right].
$}
\]

A monic polynomial $f\in\mathbb{K}[[X]][y]$ is called a
\textbf{Puiseux polynomial} if there exist
$\xi_1,\ldots,\xi_d\in\mathbb{K}[[X^*]]$ such that
\[
f=\prod_{i=1}^d(y-\xi_i).
\]
We consider Puiseux hypersurfaces $f$ such that
$\operatorname{char}(\mathbb{K})$ does not divide $\deg(f)$.

An analytic germ of dimension $n$ is \textbf{quasi-ordinary}
when it is a local covering of $\mathbb{K}^n$ unramified outside
the coordinate hyperplanes.
In the complex analytic case, the Abhyankar--Jung theorem
\cite{abhyankar1955ramification} states that quasi-ordinary
hypersurfaces are Puiseux.
Quasi-ordinary singularities have been widely studied;
see, for example,
\cite{alonso1989algorithm, MR2957197, mourtada2015polyhedral}.
For an introductory text, we refer to \cite{da2011semigrupo}.
Hironaka defined a more general class of singularities containing
the quasi-ordinary ones, which he called
\textbf{$\nu$--quasi-ordinary} \cite{hironaka2003theory};
see also \cite{cassou2011nu, aroca2024nu}.

It is natural to ask which properties can be extended from
quasi-ordinary to Puiseux hypersurfaces.
For instance, Lipman and Gau \cite{lipman1988topological}
associated to a root of a quasi-ordinary polynomial an ordered set
of $n$-tuples of nonnegative rational numbers, called
\textbf{characteristic exponents}.
Let $\lambda^{(1)},\ldots,\lambda^{(g)}$ be the set of characteristic
exponents of a quasi-ordinary root $\xi$. Then
\begin{equation}
\label{eqexponentescharact1}
\mathbb{K}((X))(\xi)
=
\mathbb{K}((X))
(X^{\lambda^{(1)}},\ldots,X^{\lambda^{(g)}}).
\end{equation}
In \cite{tornero2008kummer}, Tornero proved that if $\xi$ is a root
of a Puiseux polynomial, there exist
$\lambda^{(1)},\ldots,\lambda^{(g)}$ such that
(\ref{eqexponentescharact1}) holds.
Such a set is called a \textbf{set of distinguished exponents} of $\xi$.
In the quasi-ordinary case, it is known that the characteristic exponents
determine the topological type of the singularity.
The relationship between the topological type of a Puiseux hypersurface
and its distinguished exponents is not yet known.

One of the most important properties of a quasi-ordinary singularity
is that its normalization is \textbf{Hirzebruch--Jung}, that is,
analytically isomorphic to the zero-dimensional orbit of an affine toric
variety defined by a maximal simplicial pair
\cite{perez2000singularites, aroca2005normal, popescu2004analytical}.
This paper is devoted to showing that this property also holds for
Puiseux hypersurfaces
(see Theorem~\ref{mainthm}).

The paper is organized into four sections.
In Sections~\ref{seccionsemigrupos}--\ref{sectionpuiseux},
we review the basic notions of affine semigroups and their saturations,
introduce the ring of series supported on a pointed affine semigroup,
recall key results on toric varieties with precise references,
and present Puiseux series together with a method to compute their
distinguished exponents.
Finally, in Section~\ref{Mainsection}, we prove
(see Lemma~\ref{lemaaaaa}) that for
$\xi\in\mathbb{K}[[X^*]]$ with distinguished exponents
$\lambda^{(1)},\ldots,\lambda^{(g)}$, the normalization of
$\mathbb{K}[[X]][\xi]$ coincides with the normalization of
$\mathbb{K}[[X]][X^{\lambda^{(1)}},\ldots,X^{\lambda^{(g)}}]$,
which allows us to establish:

\textbf{Theorem }[\ref{mainthm}]
The normalization of an irreducible Puiseux hypersurface is either
nonsingular or has a Hirzebruch--Jung singularity.
Conversely, every Hirzebruch--Jung singularity is analytically isomorphic
to the normalization of a Puiseux hypersurface.

As a corollary, using a result of Popescu-Pampu
\cite{popescu2003higher}, in the complex analytic case we conclude that
the normalization of an irreducible Puiseux hypersurface is a complex
analytic normal quasi-ordinary singularity
(see Corollary~\ref{corolariofinal}).


\section{Affine Semigroups}
\label{sec3}
\label{seccionsemigrupos}

In this section we introduce the notion of an affine semigroup
and its relation to rational cones.

A \textbf{commutative semigroup} is a nonempty set equipped with an
associative and commutative binary operation.
A \textbf{monoid} is a commutative semigroup with an identity element.
A monoid is said to be \textbf{pointed} when its only invertible element
is the identity.

Let $S$ be a monoid. Given $E\subseteq S$, the monoid generated by $E$ is
\[
\resizebox{\columnwidth}{!}{$\displaystyle
\mathbb{N} E
:=
\left\{
\sum_{e \in \Lambda} c_e e
\ \Big|\
\Lambda\subseteq E \text{ finite},\
c_e \in \mathbb{Z}_{\geq 0}
\right\}
\subseteq S.
$}
\]

A monoid $S$ is said to be \textbf{finitely generated} if there exists
a finite set $E$ such that $S = \mathbb{N} E$.
A monoid $S$ is \textbf{cancellative} if for all $a,b,c\in S$,
the equality $a+b=c+b$ implies that $a=c$.
Given a cancellative monoid $S$, there is a natural way to embed $S$
into a group $\mathbb{Z}S$, analogous to the construction of the
fraction field of an integral domain.
This group is called the \textbf{Grothendieck group} of $S$.
We define $\mathbb{Z}S$ as the set of equivalence classes $e-f$
with $e,f\in S$, under the relation
\[
e-f\sim e'-f'
\iff e+f'=e'+f,
\]
with the operation
\[
(e-f)+(e'-f')=(e+e')-(f+f').
\]
Given $E\subseteq S$, we define $\mathbb{Z}E$ as the subgroup of
$\mathbb{Z}S$ generated by $E$.

A monoid $S$ is said to be \textbf{saturated} when for all
$k\in\mathbb{Z}_{>0}$ and $m\in\mathbb{Z}S$, the condition
$km\in S$ implies that $m\in S$.
The \textbf{saturation} $\bar{S}$ of $S$ is the smallest saturated
monoid containing $S$.

\begin{exa}
The monoid $S$ generated by $\{2,3\}$ is not saturated.
Indeed, taking $k=2$ and $m=1\in\mathbb{Z}S=\mathbb{Z}$,
we have $km=2\in S$ while $m=1\notin S$.
The saturation of $S$ is $\mathbb{Z}_{\geq0}$.
\end{exa}

A \textbf{lattice} is a group isomorphic to $\mathbb{Z}^r$
for some integer $r\geq1$.
A cancellative finitely generated monoid is called an
\textbf{affine semigroup} when its Grothendieck group is a lattice.
Let $S$ be an affine semigroup with Grothendieck group
$M\cong\mathbb{Z}^r$.
We write
$M_{\mathbb{R}}=\mathbb{R}\otimes_{\mathbb{Z}}M\cong\mathbb{R}^r$
for the real vector space obtained by extension of scalars.
The semigroup $S$ embeds naturally into $M_{\mathbb{R}}$
via the inclusion $M\hookrightarrow M_{\mathbb{R}}$.

A subset $\sigma$ of $M_\mathbb{R}$ of the form
\[
\resizebox{\columnwidth}{!}{$\displaystyle
\sigma=\langle E\rangle
:=
\left\{
\sum_{e\in\Lambda}c_e e
\ \Big|\
\Lambda\subseteq E\text{ finite},\
c_e\in\mathbb{R}_{\geq0}
\right\},
$}
\]
where $E$ is a finite subset of the lattice $M$, is called a
\textbf{rational polyhedral cone}, which we will from now on refer
to simply as a cone.
The set $E$ is called a \textbf{system of generators} of $\sigma$.
If the cone contains no one-dimensional linear subspace, it is said
to be \textbf{strongly convex}; and when it possesses a linearly
independent system of generators, it is called \textbf{simplicial}.
An affine semigroup $S$ is said to be \textbf{simplicial} if
$\langle S\rangle$ is simplicial.

\begin{rem}\label{Remarksat}
Let $S$ be an affine semigroup.
The saturation of $S$ is
\(
\bar{S}=\mathbb{Z}S\cap\langle S\rangle.
\)
\end{rem}

We conclude this section with some lemmas that will be useful
in the following sections.

\begin{lem}
\label{lemadca}
Let $S$ be a cancellative pointed monoid and let
$s_1,\dots,s_r\in S\setminus\{0\}$.
If $k_1,\dots,k_r,\ell_1,\dots,\ell_r\in\mathbb{Z}_{\geq0}$
satisfy $k_i\leq\ell_i$ for all $i$ and
\[
\sum_{i=1}^r k_i s_i
=
\sum_{i=1}^r \ell_i s_i,
\]
then $k_i=\ell_i$ for every $i$.
\end{lem}

\begin{proof}
As $S$ is cancellative and $k_i\leq\ell_i$ for all $i$, we obtain
\[
\sum_{i=1}^r(\ell_i-k_i)s_i=0,
\]
where all coefficients $\ell_i-k_i$ are non-negative.
Suppose that $\ell_j-k_j>0$ for some $j$.
Then
\[
s_j+
\underbrace{
\left(
\sum_{i\ne j}(\ell_i-k_i)s_i
+
(\ell_j-k_j-1)s_j
\right)
}_{\in S}
=0,
\]
so $s_j$ is invertible in $S$.
As $S$ is pointed, this forces $s_j=0$,
contradicting our hypothesis.
Hence $\ell_i-k_i=0$ for all $i$.
\end{proof}

\begin{lem}
\label{lemamultiplicK[[S]]}
Let $S$ be a pointed cancellative affine semigroup, and let
$s_1,\dots,s_r\in S\setminus\{0\}$.
For every $\gamma\in S$, the set
\[
A_\gamma
:=
\left\{
(k_1,\dots,k_r)\in(\mathbb{Z}_{\geq0})^r
\;\middle|\;
\sum_{i=1}^r k_i s_i=\gamma
\right\}
\]
is finite.
\end{lem}

\begin{proof}
Let $N_\gamma$ be the Newton polygon
\[
N_\gamma
:=
\operatorname{Conv}
\bigl(A_\gamma+(\mathbb{R}_{\geq0})^r\bigr).
\]
By Lemma~\ref{lemadca}, for every $a\in A_\gamma$,
\[
(a+(\mathbb{R}_{\geq0})^r)\cap A_\gamma=\{a\}.
\]
Hence every point of $A_\gamma$ is in a compact face of $N_\gamma$.
Consequently, $A_\gamma$ is contained in the union of the compact
faces of $N_\gamma$. Therefore, $A_\gamma$ is finite.
\end{proof}

\begin{lem}
\label{lemasimplicial}
Let $S\subset M$ be an affine semigroup whose cone
$\sigma=\langle S\rangle\subseteq M_{\mathbb{R}}$ is simplicial and
$\dim(\sigma)=n$.
There exist $k\in\mathbb{Z}_{>0}$ and a semigroup $T$ isomorphic to $S$
such that
\[
\mathbb{Z}^n
\subset
\mathbb{Z}T
\subset
\left(\tfrac{1}{k}\mathbb{Z}\right)^n
\qquad\text{and}\qquad
\langle T\rangle
=
(\mathbb{R}_{\geq0})^n.
\]
\end{lem}

\begin{proof}
Since $\sigma$ is simplicial of dimension $n$, it has a linearly
independent generating set $v_1,\dots,v_n$, unique up to positive scaling.
Since $\sigma=\langle S\rangle$, for every $i$ there must exist
$u_i\in S$ proportional to $v_i$, otherwise $\sigma$ would already be
generated by the remaining $n-1$ vectors, contradicting
$\dim\sigma=n$.
Thus $u_1,\dots,u_n\in S$ are linearly independent and generate $\sigma$.

Let $A\in M_n(\mathbb{Z})$ be the matrix whose columns are
$u_1,\dots,u_n$.
Note that $\det(A)\neq0$.
Define
\[
\varphi:\mathbb{R}^n\to\mathbb{R}^n,
\qquad
\varphi(x):=A^{-1}x.
\]
Set $T:=\varphi(S)$.
As $\varphi$ is injective, the restriction
$\varphi|_S:S\to T$ is an isomorphism of semigroups.
Since $\varphi(u_i)=e_i$ for every $i$, we have
\[
\langle T\rangle
=
\varphi(\sigma)
=
(\mathbb{R}_{\geq0})^n.
\]
Moreover, $e_i\in T$ for all $i$, so
$\mathbb{Z}^n\subseteq\mathbb{Z}T$.
Finally,
\[
A^{-1}
=
\frac{1}{\det(A)}\,\mathrm{adj}(A),
\]
so all entries of $A^{-1}$ lie in
$\frac{1}{k}\mathbb{Z}$ where $k:=|\det(A)|$.
Therefore,
\[
\mathbb{Z}T
\subseteq
\frac{1}{k}\mathbb{Z}^n.
\]
\end{proof}


\section{The Semigroup algebra and its completion}

Given a monoid $S$, we denote by $\mathbb{K}[S]$ the ring
\[
\resizebox{\columnwidth}{!}{$\displaystyle
\mathbb{K}[S]
:=
\left\{
\sum_{\alpha\in\Lambda}a_\alpha\mathbf{x}^\alpha
\ \Big|\
a_\alpha\in\mathbb{K},\
\Lambda\subseteq S,\
\#\Lambda<\infty
\right\}.
$}
\]

Let $S$ be a cancellative finitely generated pointed monoid.
We denote by $\mathbb{K}[[S]]$ the set
\[
\mathbb{K}[[S]]
:=
\left\{
\sum_{\alpha\in\Lambda}a_\alpha X^\alpha
\ \Big|\
a_\alpha\in\mathbb{K},\
\Lambda\subseteq S
\right\}.
\]

Given
$f=\sum_{\alpha\in S}a_\alpha X^\alpha$
and
$g=\sum_{\beta\in S}b_\beta X^\beta$
in $\mathbb{K}[[S]]$, define the product by
\[
fg
=
\sum_{\gamma\in S}
\left(
\sum_{\substack{\alpha,\beta\in S\\
\alpha+\beta=\gamma}}
a_\alpha b_\beta
\right)X^\gamma.
\]

To see that the inner sum is finite for every $\gamma\in S$,
let $s_1,\ldots,s_r\in S\setminus\{0\}$ be a finite generating set of $S$.
Write
\[
\alpha=\sum_{i=1}^r k_i s_i,
\qquad
\beta=\sum_{i=1}^r\ell_i s_i
\]
with $k_i,\ell_i\in\mathbb{Z}_{\geq0}$.
If $\alpha+\beta=\gamma$, then
\[
\sum_{i=1}^r(k_i+\ell_i)s_i=\gamma,
\]
so by Lemma~\ref{lemamultiplicK[[S]]} there are only finitely many
possible values for $(k_1+\ell_1,\dots,k_r+\ell_r)$.
Since
\[
0\leq k_i\leq k_i+\ell_i,
\qquad
0\leq\ell_i\leq k_i+\ell_i
\]
for all $i$, there are only finitely many possibilities for
$\alpha$ and $\beta$.
Hence the inner sum is finite, and $\mathbb{K}[[S]]$ is a ring.

\begin{rem}
\label{lemasaturado}
Let $S$ be a finitely generated cancellative pointed monoid.
The completion of $\mathbb{K}[S]$ with respect to its maximal ideal
$\mathfrak{m}=\langle X^\alpha\mid\alpha\in S,\alpha\neq0\rangle$
is $\mathbb{K}[[S]]$,
see \cite[Example 10.14, p.~181]{eisenbud2013commutative}.
\end{rem}

\begin{lem}
\label{lema}
Let $S$ be a pointed affine semigroup and let $T\subseteq S$ be a
sub-semigroup such that
$\langle T\rangle=\langle S\rangle$
and
$\mathbb{Z}T=\mathbb{Z}S$.
Let
$\gamma^{(1)},\ldots,\gamma^{(g)}\in S$
be such that $S$ is generated by $T$ together with
$\gamma^{(1)},\ldots,\gamma^{(g)}$.
Then
\[
\mathbb{K}[[S]]
=
\mathbb{K}[[T]]
\left[
X^{\gamma^{(1)}},
\ldots,
X^{\gamma^{(g)}}
\right].
\]
\end{lem}

\begin{proof}
We proceed by induction on $g$.
For $g=0$, since $S=T$, the equality holds trivially.
Assume that the statement holds for $g-1$, and let $S'$ be the
sub-semigroup of $S$ generated by $T$ together with
$\gamma^{(1)},\ldots,\gamma^{(g-1)}$.
Since $T\subseteq S'\subseteq S$, we have
\[
\langle T\rangle
\subseteq
\langle S'\rangle
\subseteq
\langle S\rangle
\]
and
\[
\mathbb{Z}T
\subseteq
\mathbb{Z}S'
\subseteq
\mathbb{Z}S.
\]
As
$\langle T\rangle=\langle S\rangle$
and
$\mathbb{Z}T=\mathbb{Z}S$,
then
\[
\langle S'\rangle=\langle T\rangle,
\qquad
\mathbb{Z}S'=\mathbb{Z}T.
\]

By the induction hypothesis,
\[
\mathbb{K}[[S']]
=
\mathbb{K}[[T]]
\left[
X^{\gamma^{(1)}},
\ldots,
X^{\gamma^{(g-1)}}
\right].
\]
Clearly,
\[
\mathbb{K}[[S']]
\left[
X^{\gamma^{(g)}}
\right]
\subseteq
\mathbb{K}[[S]].
\]

Conversely, let $f\in\mathbb{K}[[S]]$.
Since $S$ is generated by $S'$ and $\gamma^{(g)}$,
every $\alpha\in S$ is of the form
\[
\alpha=\alpha'+m\gamma^{(g)}
\]
for some $\alpha'\in S'$ and
$m\in\mathbb{Z}_{\geq0}$.
Thus $f$ can be written as
\[
f
=
\sum_{\substack{\alpha\in S'\\m\in\mathbb{Z}_{\geq0}}}
a_{\alpha,m}
X^\alpha
\left(X^{\gamma^{(g)}}\right)^m.
\]

Since
\[
\gamma^{(g)}
\in
S
\subseteq
\langle S\rangle
=
\langle T\rangle
\]
and
\[
\gamma^{(g)}
\in
S
\subseteq
\mathbb{Z}S
=
\mathbb{Z}T,
\]
we have
\[
\gamma^{(g)}
\in
\mathbb{Z}T\cap\langle T\rangle
=
\overline{T}.
\]
Hence there exists $k\in\mathbb{Z}_{>0}$ such that
$k\gamma^{(g)}\in T$.

Write
\[
m=j+k\ell,
\qquad
0\leq j<k,
\qquad
\ell\in\mathbb{Z}_{\geq0}.
\]
Then
\[
\resizebox{\columnwidth}{!}{$\displaystyle
f
=
\sum_{j=0}^{k-1}
X^{j\gamma^{(g)}}
\sum_{\ell\geq0}
\left(
\sum_{\alpha\in S'}
a_{\alpha,j+k\ell}X^\alpha
\right)
\left(X^{\gamma^{(g)}}\right)^{k\ell}.
$}
\]
Since
\[
\left(X^{\gamma^{(g)}}\right)^{k\ell}
\in
\mathbb{K}[[T]]
\subseteq
\mathbb{K}[[S']],
\]
it follows that each inner sum belongs to $\mathbb{K}[[S']]$.
Therefore,
\[
f
\in
\mathbb{K}[[S']]
\left[
X^{\gamma^{(g)}}
\right].
\]
Hence
\[
\resizebox{\columnwidth}{!}{$\displaystyle
\mathbb{K}[[S]]
=
\mathbb{K}[[S']]
\left[
X^{\gamma^{(g)}}
\right]
=
\mathbb{K}[[T]]
\left[
X^{\gamma^{(1)}},
\ldots,
X^{\gamma^{(g)}}
\right].
\qedhere
$}
\]
\end{proof}


\begin{defn}
A monic polynomial
$f\in\mathbb{K}[[X]][y]$
is called a \textbf{Puiseux polynomial}
when it splits in the ring of \textbf{Puiseux series}
\[
\mathbb{K}[[X^*]]
:=
\bigcup_{m\in\mathbb{Z}_{>0}}
\mathbb{K}
\left[
\left[
x_1^{1/m},\ldots,x_n^{1/m}
\right]
\right].
\]
\end{defn}

\begin{rem}
\label{isomK[[x]]/fyK[[x]][xi]}
Let
$f\in\mathbb{K}[[X]][y]$
be an irreducible Puiseux polynomial and let
$\xi\in\mathbb{K}[[X^*]]$
be a root of $f$.
There exists an isomorphism between
\[
\mathbb{K}[[X]][y]/\langle f\rangle
\]
and
\[
\mathbb{K}[[X]][\xi].
\]
\end{rem}

\begin{defn}
\label{defDisting}
Let $\xi$ be a Puiseux series.
A set
\[
E=
\left\{
\lambda^{(1)},\ldots,\lambda^{(g)}
\right\}
\subseteq
\operatorname{Supp}(\xi)
\subseteq
\left(\frac{1}{k}\mathbb{Z}_{\geq0}\right)^n
\]
is a set of \textbf{distinguished exponents} of $\xi$ if
\begin{equation}
\label{eqexponentescharact}
\mathbb{K}((X))(\xi)
=
\mathbb{K}((X))
\left(
X^{\lambda^{(1)}},
\ldots,
X^{\lambda^{(g)}}
\right).
\end{equation}
\end{defn}

In \cite{tornero2008kummer}, Tornero proves that every Puiseux series
$\xi$ whose minimal polynomial $f$ satisfies
$\operatorname{char}(\mathbb{K})\nmid\deg(f)$ admits a set of
distinguished exponents\footnote{
In \cite{tornero2008kummer}, the distinguished exponents are defined
as elements of the set
$\{k\lambda^{(1)},\ldots,k\lambda^{(g)}\}
\subset k\,\operatorname{Supp}(\xi)
\subset(\mathbb{Z}_{\geq0})^n$,
where $\lambda^{(1)},\ldots,\lambda^{(g)}$
satisfy equation (\ref{eqexponentescharact}).
Our definition differs by a factor of $k$.
}.

Given a vector
$\omega\in(\mathbb{R}_{\geq0})^n$
with rationally independent coordinates, consider the total order
$\leq_\omega$ on $(\mathbb{Q}_{\geq0})^n$ defined by
\[
\lambda\leq_\omega\beta
\quad\text{if and only if}\quad
\lambda\cdot\omega\leq\beta\cdot\omega.
\]
A set of distinguished exponents of $\xi$ can be constructed explicitly
via the following recursive procedure.
Set
\[
\lambda^{(1)}
=
\min_{\leq_\omega}
\left\{
\operatorname{Supp}(\xi)\setminus\mathbb{Z}^n
\right\}
\]
and
\begin{equation}
\label{eq22}
\lambda^{(i)}
=
\min_{\leq_\omega}
\left\{
\operatorname{Supp}(\xi)
\setminus
\left(
\mathbb{Z}^n+
\sum_{j=1}^{i-1}\lambda^{(j)}\mathbb{Z}
\right)
\right\},
\end{equation}
for $i=2,\ldots,g$, where
\[
g=\max\{i\mid\lambda^{(i)}\neq\infty\}.
\]
For more details we refer to \cite{tornero2008kummer}.

In the rest of this paper we consider Puiseux hypersurfaces $f$
such that
$\operatorname{char}(\mathbb{K})\nmid\deg(f)$.


\section{Normal Toric Varieties}
\label{sec4}
\label{sectnormalsatu}

An \textbf{affine toric variety} is an irreducible affine variety $V$
containing a torus
$T\simeq(\mathbb{K}^*)^n$
as a Zariski open subset, such that the algebraic action of $T$
on itself extends to an algebraic action of $T$ on $V$.

\begin{prop}
\label{specdesemigrupo}
Let $S$ be a monoid.
Then $S$ is an affine semigroup if and only if
$\mathbb{K}[S]$ is the coordinate ring over $\mathbb{K}$
of an affine toric variety.
\end{prop}

See
\cite[Propositions 1.1.14 and 1.3.5]{cox2024toric},
\cite[Theorem 1, Chapter 1, Section 1]{Kempf}.

Given a rational lattice $M\simeq\mathbb{Z}^n$
and a convex polyhedral cone
$\sigma\subseteq M_{\mathbb{R}}$,
by Proposition~\ref{specdesemigrupo} there is an associated
affine toric variety
\[
T_{\sigma,M}
=
\operatorname{Spec}\mathbb{K}[\sigma\cap M].
\]
In the following, we consider $\sigma$ to be a strongly convex
rational cone of maximal dimension in $\mathbb{R}^n$.

\begin{defn}
An affine variety $\mathbb{X}$ is said to be \textbf{normal}
when its coordinate ring $\mathcal{A}(\mathbb{X})$ is integrally
closed in its field of fractions.
The \textbf{normalization} $\mathcal{N}(I)$ of a domain $I$
is the integral closure of $I$ in its field of fractions.
\end{defn}

\begin{exa}
The toric variety
$\mathbb{X}=V(x^3-y^2)$
is not normal: the monic polynomial
\[
p(z)=z^2-\overline{x}\in\mathcal{A}(\mathbb{X})[z]
\]
has
$\alpha=\frac{\overline{y}}{\overline{x}}$
as a root; this element lies in the field of fractions of
$\mathcal{A}(\mathbb{X})$ but not in
$\mathcal{A}(\mathbb{X})$ itself.
\end{exa}

\begin{lem}
\label{lemasaturado2}
Let $\bar{S}$ be the saturation of the affine semigroup $S$.
The normalization of $\mathbb{K}[S]$ is $\mathbb{K}[\bar{S}]$.
\end{lem}

See
\cite[Propositions 1.3.5 and 1.3.8]{cox2024toric},
\cite[Theorem 1, Chapter 1, Section 1]{Kempf}.

\begin{defn}
\label{maximalespecial}
Let $S$ be a pointed affine semigroup.
The maximal ideal of $\mathbb{K}[S]$ generated by
\[
\{x^\alpha\mid\alpha\in S\setminus\{0\}\}
\]
is referred to as the \textbf{special ideal}.
The point of the toric variety that corresponds to the special maximal
ideal is called the \textbf{special point}.
\end{defn}


\label{sectionpuiseux}
\label{sec6}

\section{Main Results}
\label{Mainsection}
\label{sec7}

\begin{lem}
\label{lemaaaaa}
Let
$\lambda^{(1)},\ldots,\lambda^{(g)}
\in
\big(\tfrac{1}{k}\mathbb{Z}_{\geq0}\big)^n$
be a set of distinguished exponents of a root $\xi$
of a Puiseux polynomial.
Let $\bar{S}$ be the saturation of the affine semigroup $S$
generated by
$\lambda^{(1)},\ldots,\lambda^{(g)}$
and
$(\mathbb{Z}_{\geq0})^n$.
We have
\[
\mathcal{N}
\left(
\mathbb{K}[[X]][\xi]
\right)
=
\mathbb{K}[[\bar{S}]].
\]
\end{lem}

\begin{proof}
Lemma~\ref{lema} states that
\[
\mathbb{K}[[X]]
\left[
X^{\lambda^{(1)}},
\ldots,
X^{\lambda^{(g)}}
\right]
=
\mathbb{K}[[S]].
\]
Since excellent rings satisfy that normalization and completion commute
(see \cite[Proposition 7.8.3.(vii)]{egaiv1964grothendieck}),
it follows from Lemma~\ref{lemasaturado2} that
\[
\mathcal{N}(\mathbb{K}[[S]])
=
\mathbb{K}[[\bar{S}]],
\]
hence
\[
\resizebox{\columnwidth}{!}{$\displaystyle
\mathcal{N}
\left(
\mathbb{K}[[X]]
\left[
X^{\lambda^{(1)}},
\ldots,
X^{\lambda^{(g)}}
\right]
\right)
=
\mathcal{N}(\mathbb{K}[[S]])
=
\mathbb{K}[[\bar{S}]].
$}
\]

For each distinguished exponent $\lambda^{(i)}$,
the monomial $X^{\lambda^{(i)}}$ is a root of the monic polynomial
\[
y^k-X^{k\lambda^{(i)}}
\in
\mathbb{K}[[X]][y]
\subseteq
\mathbb{K}[[X]][\xi][y].
\]
By equation~\eqref{eqexponentescharact},
$\mathbb{K}[[X]][\xi]$ and $\mathbb{K}[[S]]$
have the same field of fractions.
Consequently,
\[
X^{\lambda^{(i)}}
\in
\operatorname{Frac}
\left(
\mathbb{K}[[X]][\xi]
\right).
\]
Therefore,
\[
X^{\lambda^{(i)}}
\in
\mathcal{N}
\left(
\mathbb{K}[[X]][\xi]
\right).
\]
Hence,
\[
\mathbb{K}[[\bar{S}]]
\subseteq
\mathcal{N}
\left(
\mathbb{K}[[X]][\xi]
\right).
\]

On the other hand, since
$\xi\in\operatorname{Frac}(\mathbb{K}[[S]])$
is a root of a monic polynomial with coefficients in
$\mathbb{K}[X]\subseteq\mathbb{K}[[S]]$,
we obtain that
$\xi\in\mathbb{K}[[\bar{S}]]$,
which completes the proof.
\end{proof}

Two singularities $p\in\mathbb{X}$ and $q\in\mathbb{Y}$
are said to be \textbf{analytically isomorphic} when there exists
an isomorphism of $\mathbb{K}$-algebras
\[
\widehat{\mathcal{O}}_p
\simeq
\widehat{\mathcal{O}}_q.
\]
A singular point of a normal irreducible algebroid hypersurface is
called a \textbf{Hirzebruch-Jung singularity} when it is analytically
isomorphic to the special point of an affine toric variety
$\mathbb{K}[S]$ such that $S$ is simplicial.

\begin{thm}
\label{mainthm}
Let $\mathbb{K}$ be an algebraically closed field and
$f\in\mathbb{K}[[X]][y]$
an irreducible Puiseux polynomial with
\[
\operatorname{char}(\mathbb{K})
\nmid
\deg(f).
\]
Then the normalization of the hypersurface defined by $f$
is either non-singular or has a Hirzebruch-Jung singularity.
Conversely, every Hirzebruch-Jung singularity is analytically
isomorphic to the normalization of a Puiseux hypersurface.
\end{thm}

\begin{proof}
Let
$f\in\mathbb{K}[[X]][y]$
be an irreducible monic Puiseux polynomial and $\xi$ a root of $f$.
Remark~\ref{isomK[[x]]/fyK[[x]][xi]} gives
\[
\mathbb{K}[[X]][y]/\langle f\rangle
\simeq
\mathbb{K}[[X]][\xi].
\]
Choose a set $E$ of distinguished exponents of $\xi$
as in \eqref{eq22} and let $S$ be the affine semigroup generated by
$E$ and $(\mathbb{Z}_{\geq0})^n$.
Remark~\ref{Remarksat} states that the saturation $\bar{S}$ is the
intersection of $\mathbb{Z}S$ and the simplicial cone
$(\mathbb{R}_{\geq0})^n$.
By Proposition~\ref{specdesemigrupo},
$\mathbb{K}[\bar{S}]$ is the coordinate ring of an affine toric variety,
which is normal by Lemma~\ref{lemasaturado2} and its completion is
$\mathbb{K}[[\bar{S}]]$ by Remark~\ref{lemasaturado}.
Lemma~\ref{lemaaaaa} implies that
\[
\mathcal{N}
\left(
\mathbb{K}[[X]][\xi]
\right)
=
\mathbb{K}[[\bar{S}]].
\]
Hence,
$\mathcal{N}
(\mathbb{K}[[X]][y]/\langle f\rangle)$
is either non-singular or has a Hirzebruch-Jung singularity.

Conversely, let us consider a Hirzebruch-Jung singularity.
By Lemma~\ref{lemasimplicial}, it must be of the form
\[
\operatorname{Spec}(\mathbb{K}[S]),
\]
where
\[
S=M\cap(\mathbb{R}_{\geq0})^n
\]
is an affine semigroup and $M$ a subgroup of $\mathbb{Z}^n$.
Let $m_i\in\mathbb{Z}_{\geq0}$ be defined by
\[
m_i e_i
=
\min
\{
M\cap\langle e_i\rangle\setminus\{0\}
\},
\]
where $e_i$ is the $i$-th canonical basis vector.
Define
\[
\resizebox{\columnwidth}{!}{$\displaystyle
M'
=
\left\{
\left(
\frac{\alpha_1}{m_1},
\ldots,
\frac{\alpha_n}{m_n}
\right)
\ \middle|\
\alpha\in M
\right\},
\qquad
S'
=
M'\cap(\mathbb{R}_{\geq0})^n.
$}
\]
We have
\begin{equation}
\label{hola}
\mathbb{K}[S']
\simeq
\mathbb{K}[S].
\end{equation}
Since
$\operatorname{Spec}(\mathbb{K}[S])$ is singular,
$M'\neq\mathbb{Z}^n$.

Choose
$\omega\in(\mathbb{R}_{\geq0})^n$
with rationally independent coordinates and construct
\[
\lambda^{(1)}
=
\min_{\leq_\omega}
\{S'\setminus\mathbb{Z}^n\},
\]
\[
\resizebox{\columnwidth}{!}{$\displaystyle
\lambda^{(i)}
=
\min_{\leq_\omega}
\left\{
S'
\setminus
\left(
\mathbb{Z}^n+
\sum_{j=1}^{i-1}\mathbb{Z}\lambda^{(j)}
\right)
\right\},
\qquad
i=2,\ldots,g,
$}
\]
with
\[
g
=
\max
\{
i\mid
\lambda^{(i)}\neq\infty
\}.
\]
Then $S'$ is the saturation of the affine semigroup generated by
$E=\{\lambda^{(1)},\ldots,\lambda^{(g)}\}$
and $(\mathbb{Z}_{\geq0})^n$.
Let
\[
\xi
=
\sum_{i=1}^g X^{\lambda^{(i)}}
\]
and $f$ its minimal polynomial.
Then
\[
\resizebox{\columnwidth}{!}{$\displaystyle
\mathcal{N}
\left(
\mathbb{K}[[X]][y]/\langle f\rangle
\right)
\simeq
\mathcal{N}
\left(
\mathbb{K}[[X]][\xi]
\right)
\simeq
\mathbb{K}[[S']]
\simeq
\mathbb{K}[[S]],
$}
\]
where the first isomorphism is given by
Remark~\ref{isomK[[x]]/fyK[[x]][xi]},
the second isomorphism follows from Lemma~\ref{lemaaaaa},
and the third one is by~\eqref{hola}.
This proves that the given Hirzebruch-Jung singularity is analytically
isomorphic to the normalization of the Puiseux hypersurface defined by $f$.
\end{proof}

\begin{exa}
\label{examLipman}
Lipman proved in the complex case that every normal quasi-ordinary
singularity is isomorphic to a germ of the form
$y^k-x_1\cdots x_n$ for some $k\geq2$,
see \cite[Remark 7.3.2]{lipman1988topological}.
Then every normal quasi-ordinary singularity is analytically isomorphic
to the affine toric variety
$\operatorname{Spec}(\mathbb{C}[S])$,
where $S$ is the affine semigroup generated by
$(1/k,\ldots,1/k)$ and $(\mathbb{Z}_{\geq0})^n$.
\end{exa}

\begin{exa}
The roots of
\[
f=(z^2-x-y)^2-4xy
\in
\mathbb{K}[[x,y]][z]
\]
are the Puiseux series
$\pm x^{1/2}\pm y^{1/2}$,
which are not quasi-ordinary, see \cite{Ayala2025}.
Following Tornero's algorithm, we obtain distinguished exponents
$(1/2,0)$ and $(0,1/2)$.
The affine semigroup $S$ generated by
$(\mathbb{Z}_{\geq0})^2$
and these exponents is already saturated.
The normalization of the hypersurface defined by $f$ is analytically
isomorphic to the non-singular toric surface
$\mathbb{K}[S]\simeq\mathbb{K}^2$.
\end{exa}

In the complex-analytic case,
Popescu-Pampu \cite[Proposition 3.5]{popescu2003higher}
proved that Hirzebruch-Jung singularities are precisely the normal
quasi-ordinary singularities.
Hence, we have:

\begin{cor}
\label{corolariofinal}
The germ of a normal quasi-ordinary singularities is the normalizations
of complex analytic irreducible Puiseux hypersurfaces.
\end{cor}

\begin{exa}
Let
\[
\xi=x+x^{2/3}y^{1/3}+y
\]
and $f$ be its minimal polynomial.
The only distinguished exponent of $\xi$ is $(2/3,1/3)$.
The saturation of the affine semigroup $S$ generated by
$(2/3,1/3)$ and $(\mathbb{Z}_{\geq0})^2$ is generated by
\[
\{(1,0),(1/3,2/3),(2/3,1/3),(0,1)\}.
\]
By Theorem~\ref{mainthm},
\[
\mathcal{N}
\left(
\mathbb{C}[[x]][y]/\langle f\rangle
\right)
\]
is analytically isomorphic to the toric variety
\[
\operatorname{Spec}
\left(
\mathbb{C}[\bar{S}]
\right),
\]
whose canonical embedding in $\mathbb{C}^4$ is the Veronese variety
\[
\operatorname{Spec}
\left(
\mathbb{C}[\bar{S}]
\right)
=
V(xw-yz,y^2-xz,z^2-yw).
\]
The projection
\[
\pi:(x,y,z,w)\mapsto(x,w)
\]
is a quasi-ordinary projection.
Therefore, starting from a Puiseux hypersurface, we obtain its
normalization, which is a non-hypersurface quasi-ordinary singularity.
\end{exa}


{\footnotesize\setlength{\bibitemsep}{0.35\baselineskip}\printbibliography[title={References}]}

@article{abhyankar1955ramification,
  author = {Abhyankar, S.},
  title = {On the ramification of algebraic functions},
  journal = {Amer. J. Math.},
  fjournal = {American Journal of Mathematics},
  volume = {77},
  year = {1955},
  pages = {575--592},
  issn = {0002-9327,1080-6377},
  mrclass = {14.0X},
  mrnumber = {71851},
  mrreviewer = {H. T. Muhly},
  doi = {10.2307/2372643},
  url = {https://doi.org/10.2307/2372643}
}

@incollection{aroca2005normal,
  author = {Aroca, F. and Snoussi, J.},
  title = {Normal quasi-ordinary singularities},
  booktitle = {Singularit\'es Franco-Japonaises},
  series = {S\'emin. Congr.},
  volume = {10},
  pages = {1--10},
  publisher = {Soc. Math. France, Paris},
  year = {2005},
  isbn = {2-85629-166-X},
  mrclass = {32S05 (14B05)},
  mrnumber = {2145944},
  mrreviewer = {J. Stevens}
}

@book{cox2024toric,
  author = {Cox, D. A. and Little, J. B. and Schenck, H. K.},
  title = {Toric varieties},
  series = {Graduate Studies in Mathematics},
  volume = {124},
  publisher = {American Mathematical Society, Providence, RI},
  year = {2011},
  pages = {xxiv+841},
  isbn = {978-0-8218-4819-7},
  mrclass = {14M25 (05A15 05E45 52B12)},
  mrnumber = {2810322},
  mrreviewer = {I. Arzhantsev},
  doi = {10.1090/gsm/124},
  url = {https://doi.org/10.1090/gsm/124}
}

@book{eisenbud2013commutative,
  author = {Eisenbud, D.},
  title = {Commutative algebra},
  series = {Graduate Texts in Mathematics},
  volume = {150},
  note = {With a view toward algebraic geometry},
  publisher = {Springer-Verlag, New York},
  year = {1995},
  pages = {xvi+785},
  isbn = {0-387-94268-8},
  mrclass = {13-01 (14A05)},
  mrnumber = {1322960},
  mrreviewer = {M. Miller},
  doi = {10.1007/978-1-4612-5350-1},
  url = {https://doi.org/10.1007/978-1-4612-5350-1}
}

@article{perez2000singularites,
  author = {Gonz\'alez P\'erez, P. D.},
  title = {Singularit\'es quasi-ordinaires toriques et poly\`edre de Newton du discriminant},
  journal = {Canad. J. Math.},
  fjournal = {Canadian Journal of Mathematics. Journal Canadien de Math\'ematiques},
  volume = {52},
  year = {2000},
  number = {2},
  pages = {348--368},
  issn = {0008-414X,1496-4279},
  mrclass = {14M25 (14B05)},
  mrnumber = {1755782},
  mrreviewer = {M. Lejeune-Jalabert},
  doi = {10.4153/CJM-2000-016-8},
  url = {https://doi.org/10.4153/CJM-2000-016-8}
}

@article{lipman1988topological,
  author = {Lipman, J. and Gau, Y.-N.},
  title = {Topological invariants of quasi-ordinary singularities. Embedded topological classification of quasi-ordinary singularities},
  journal = {Mem. Amer. Math. Soc.},
  fjournal = {Memoirs of the American Mathematical Society},
  volume = {74},
  year = {1988},
  number = {388},
  pages = {1--107},
  issn = {0065-9266,1947-6221},
  mrclass = {14B05 (32C40 57Q45)},
  mrnumber = {954947},
  mrreviewer = {U. Karras},
  doi = {10.1090/memo/0388},
  url = {https://doi.org/10.1090/memo/0388}
}

@article{popescu2003higher,
  author = {Popescu-Pampu, P.},
  title = {On higher dimensional {H}irzebruch-{J}ung singularities},
  journal = {Rev. Mat. Complut.},
  fjournal = {Revista Matem\'atica Complutense},
  volume = {18},
  year = {2005},
  number = {1},
  pages = {209--232},
  issn = {1139-1138,1988-2807},
  mrclass = {32S10 (14M25)},
  mrnumber = {2135539},
  mrreviewer = {M. Tib\u ar}
}

@article{popescu2004analytical,
  author = {Popescu-Pampu, P.},
  title = {On the analytical invariance of the semigroups of a quasi-ordinary hypersurface singularity},
  journal = {Duke Math. J.},
  fjournal = {Duke Mathematical Journal},
  volume = {124},
  year = {2004},
  number = {1},
  pages = {67--104},
  issn = {0012-7094,1547-7398},
  mrclass = {32S10 (14M25)},
  mrnumber = {2072212},
  mrreviewer = {L. McEwan},
  doi = {10.1215/S0012-7094-04-12413-3},
  url = {https://doi.org/10.1215/S0012-7094-04-12413-3}
}

@article{tornero2008kummer,
  author = {Tornero, J. M.},
  title = {On {K}ummer extensions of the power series field},
  journal = {Math. Nachr.},
  fjournal = {Mathematische Nachrichten},
  volume = {281},
  year = {2008},
  number = {10},
  pages = {1511--1519},
  issn = {0025-584X,1522-2616},
  mrclass = {13F25 (13J05)},
  mrnumber = {2454947},
  mrreviewer = {J. L. Johnson},
  doi = {10.1002/mana.200510692},
  url = {https://doi.org/10.1002/mana.200510692}
}

@book{Kempf,
  author = {Kempf, G. and Knudsen, F. and Mumford, D and Saint-Donat, B},
  title = {Toroidal Embeddings I},
  publisher = {Springer Verlag},
  year = {1973}
}

@article{mourtada2015polyhedral,
  author = {Mourtada, H. and Schober, B.},
  title = {A polyhedral characterization of quasi-ordinary singularities},
  journal = {arXiv preprint arXiv:1512.07507},
  year = {2015}
}

@article{MR2957197,
  author = {Budur, N. and Gonz\'alez-P\'erez, P. D. and Gonz\'alez Villa, M.},
  title = {Log canonical thresholds of quasi-ordinary hypersurface singularities},
  journal = {Proc. Amer. Math. Soc.},
  fjournal = {Proceedings of the American Mathematical Society},
  volume = {140},
  year = {2012},
  number = {12},
  pages = {4075--4083},
  issn = {0002-9939,1088-6826},
  mrclass = {14B05 (32S45)},
  mrnumber = {2957197},
  mrreviewer = {A. Dimca},
  doi = {10.1090/S0002-9939-2012-11416-9},
  url = {https://doi.org/10.1090/S0002-9939-2012-11416-9}
}

@article{hironaka2003theory,
  author = {Hironaka, H.},
  title = {Theory of infinitely near singular points},
  journal = {Journal of the Korean Mathematical Society},
  volume = {40},
  number = {5},
  pages = {901--920},
  year = {1974},
  publisher = {Korean Mathematical Society}
}

@article{aroca2024nu,
  author = {Aroca, F. and Tornero, J. M.},
  title = {On {$\nu$}-quasiordinary surface singularities and their resolution},
  journal = {Mediterr. J. Math.},
  fjournal = {Mediterranean Journal of Mathematics},
  volume = {21},
  year = {2024},
  number = {5},
  pages = {Paper No. 168, 17},
  issn = {1660-5446,1660-5454},
  mrclass = {14H20 (32S25)},
  mrnumber = {4782238},
  mrreviewer = {Krasi\'nski, T.},
  doi = {10.1007/s00009-024-02709-x},
  url = {https://doi.org/10.1007/s00009-024-02709-x}
}

@incollection{cassou2011nu,
  author = {Artal Bartolo, E. and Cassou-Nogu\`es, P. and Luengo, I. and Melle Hern\'andez, A.},
  title = {On {$\nu$}-quasi-ordinary power series: factorization, {N}ewton trees and resultants},
  booktitle = {Topology of algebraic varieties and singularities},
  series = {Contemp. Math.},
  volume = {538},
  pages = {321--343},
  publisher = {Amer. Math. Soc., Providence, RI},
  year = {2011},
  isbn = {978-0-8218-4890-6},
  mrclass = {32S70 (14B05)},
  mrnumber = {2777828},
  mrreviewer = {Zuo, H. Q.},
  doi = {10.1090/conm/538/10610},
  url = {https://doi.org/10.1090/conm/538/10610}
}

@incollection{alonso1989algorithm,
  author = {Alonso, M. E. and Luengo, I. and Raimondo, M.},
  title = {An algorithm on quasi-ordinary polynomials},
  booktitle = {Applied algebra, algebraic algorithms and error-correcting codes ({R}ome, 1988)},
  series = {Lecture Notes in Comput. Sci.},
  volume = {357},
  pages = {59--73},
  publisher = {Springer, Berlin},
  year = {1989},
  isbn = {3-540-51083-4},
  mrclass = {12E99 (12-04 13B25)},
  mrnumber = {1008493},
  mrreviewer = {Cucker, F.},
  doi = {10.1007/3-540-51083-4\_48},
  url = {https://doi.org/10.1007/3-540-51083-4_48}
}

@article{Ayala2025,
  author = {Ayala, A.},
  title = {Una extensión del concepto de monomios y exponentes característicos y su relación con el polígono de Newton del discriminante / tesis que para obtener el grado de Doctorado en Ciencias Matemáticas; tutor principal de tesis Fuensanta Aroca Bisquert. PhD thesis, Universidad Nacional Autónoma de México.},
  year = {2025}
}

@mastersthesis{da2011semigrupo,
  title={O semigrupo de uma hipersuperf{\'\i}cie quase ordin{\'a}ria},
  author={da Silva Lemes, E.},
  year={2011},
  school={Brasil}
}

@article{egaiv1964grothendieck,
  title={ El{\'e}ments de g{\'e}om{\'e}trie alg{\'e}brique. IV. Etude locale des schemas et des morphismes de schemas. II},
  author={A. Grothendieck},
  journal={I, Publ. Math. Inst. Hautes {\'E}tudes Sci},
  volume={20},
  pages={5--251},
  year={1965}
}


\section*{Declarations}

\textbf{Competing Interests:}
The authors have no relevant financial or non-financial interests
to disclose.

\end{multicols}
\section*{Afiliations}
\end{document}